\documentclass[12pt]{amsart}

\usepackage{amssymb}
\usepackage{stmaryrd}
\usepackage{fullpage}
\usepackage{amscd}
\usepackage[all]{xy}
\usepackage{comment}
\usepackage{mathrsfs}
\usepackage{longtable}

\usepackage{tikz}
\usetikzlibrary{positioning, calc}

\usepackage{color}
\usepackage{hyperref}
\usepackage{colonequals}

\numberwithin{equation}{section}
\newtheorem{theorem}[equation]{Theorem}
\newtheorem{lemma}[equation]{Lemma}
\newtheorem{corollary}[equation]{Corollary}
\newtheorem{proposition}[equation]{Proposition}

\theoremstyle{definition}
\newtheorem{definition}[equation]{Definition}

\theoremstyle{remark}
\newtheorem{remark}[equation]{Remark}

\newcommand{\CC}{\mathbb{C}}
\newcommand{\FF}{\mathbb{F}}
\newcommand{\F}{\mathbb{F}}

\newcommand{\QQ}{\mathbb{Q}}
\newcommand{\QQbar}{\overline{\mathbb{Q}}}
\newcommand{\RR}{\mathbb{R}}
\newcommand{\ZZ}{\mathbb{Z}}

\newcommand{\Arg}{\operatorname{Arg}}

\newcommand{\End}{\operatorname{End}}

\newcommand{\Gbar}{\overline{G}}
\newcommand{\supp}{\operatorname{supp}}

\newcommand{\Fr}{\operatorname{Fr}}

\newcommand{\Hom}{\operatorname{Hom}}

\newcommand{\Gal}{\operatorname{Gal}}
\newcommand{\GL}{\operatorname{GL}}

\newcommand{\rank}{\operatorname{rank}}

\begin{document}

%%%%%%%%%%%%%%%%%%%%%%%%%%%%%%%%%%%%%%%%%%%%%%%%%%%%
%%%%%%%%%%%%%% TITLE %%%%%%%%%%%%%%%%%%%%%%%%%%%%%%%%
%%%%%%%%%%%%%%%%%%%%%%%%%%%%%%%%%%%%%%%%%%%%%%%%%%%%

 \title{Angle Ranks of Abelian Varieties}

\author{Taylor Dupuy}
\author{Kiran S. Kedlaya}
\author{David Zureick-Brown}

\date{}
\thanks{Kedlaya was supported by NSF (DMS-1802161, DMS-2053473) and UC San Diego (Warschawski Professorship). Zureick-Brown was supported by NSF (DMS-1555048).}

\begin{abstract}
Using the formalism of Newton hyperplane arrangements, we resolve the open questions regarding angle rank left over from work of the first
two authors with Roe and Vincent.
As a consequence we end up generalizing theorems of Lenstra--Zarhin and Tankeev proving several new cases of the Tate conjecture for abelian varieties over finite fields. 
We also obtain an effective version of a recent theorem of Zarhin bounding the heights of coefficients in multiplicative relations among Frobenius eigenvalues.
\end{abstract}

\maketitle
\tableofcontents

\section{Introduction}

Let $A$ be an absolutely simple $g$-dimensional abelian variety over a finite field $\FF_q$ where $q$ is a power of the prime $p$.
Let $P(T)=\det(T-\Fr_q\vert H^1(A))$ be the characteristic polynomial of the Frobenius acting on the first $\ell$-adic \'etale cohomology for $\ell$ coprime to $q$. 
For future reference we will write 
 $$P(T)=(T-\alpha_1)(T-\alpha_2)\cdots (T-\alpha_{2g})$$ 
and index the roots $\alpha_i$ so that $\overline{\alpha_i} = \alpha_{i+g}=q/\alpha_i$ where $1\leq i \leq g$. 
Also for future reference, we will also refer to the Galois group $G$ of $P(T)$ as the Galois group of $A$. 
Since $P(T)$ has rational coefficients, the Galois group $G$ is over $\QQ$.
We remark that $G \subset \FF_2^g\rtimes S_g$ (i.e., the Weyl group of $\mathrm{Sp}_{2g}$)
and we call the kernel of the projection to $S_g$ the \emph{code} of the abelian variety and view it as a subgroup of $\FF_2^g$. 
This will appear again in one of our results.

We are interested in the multiplicative relations among the eigenvalues of Frobenius. More precisely, we are interested in finding all of the tuples $(e_1,e_2,\ldots,e_{2g})\in \ZZ^{2g}$ such that 
\begin{equation}\label{eqn:basic-relation}
  \alpha_1^{e_1}\alpha_2^{e_2} \cdots \alpha_{2g}^{e_{2g}} = \zeta,
\end{equation}
where $\zeta$ is some root of unity. 
Roughly speaking, one might already see that this has something to do with the Tate conjecture, as relations among eigenvalues have to do with relations among eigenvectors, which seem to be related to Tate classes --- one can also see that Galois acts on the right hand side via a cyclotomic character. 
These relations are indeed related to Tate classes.
We will return to this after defining the angle rank.

Fix an algebraic closure $\QQ \subset \overline{\QQ}$. Consider the group $\Gamma \subset \QQbar^{\times}$ generated by the roots of $P(T)$ under multiplication. 
\begin{definition}
	We define the \emph{angle rank} $\delta$ of $A$ by $\delta=\rank(\Gamma)-1$.  
\end{definition}
The angle rank is a number between $0$ and $g$ which reflects how many multiplicative relations there are among eigenvalues of the Frobenius; a small angle rank corresponds to a large number of relations among the eigenvalues. The angle rank also appears in the study of $p$-adic absolute values of $\#A(\FF_{q^r})$ (see for example Mueller \cite{Mueller2013}).
The name \emph{angle rank} comes from the fact that one could also define $\delta$ as the dimension of the $\QQ$-vector subspace of $\RR/\QQ$ spanned by the images of $\Arg(\alpha)/2\pi \in \RR$ where $\alpha$ runs over the Frobenius eigenvalues. 

\begin{remark}[Angle rank and exotic Tate classes]
If the angle rank is maximal, i.e.~$\delta(A)=g$, then  there are no nontrivial multiplicative relations of the type described in equation \eqref{eqn:basic-relation}. 
The condition $\delta(A)=g$ also implies that $A$ is \emph{neat}, meaning that all of the Tate classes are generated in codimension 1. 
This condition also implies the Tate conjecture holds for $A$: since Tate proved the Tate conjecture in codimension 1, any result that implies $\delta(A)=g$ also implies that the Tate conjecture holds for all powers of $A$. 
There is also a converse to this: when the angle rank is non-maximal, i.e.~$\delta(A)<g$,  there exists an exotic Tate class on some power of $A$; see \cite[Elementary Lemma 6.4]{Zarhin1993}, \cite{Zarhin2015}. 
\end{remark}

The first two authors, together with Roe and Vincent, undertook a systematic study of how angle rank is related to the Newton polygon and Galois group of $P(T)$ after developing a database of isogeny classes of abelian varieties for the LMFDB \cite{Dupuy2021}. 
Although much data was produced, and we found counterexamples to a conjecture about angle rank \cite{Dupuy2021b} (explicitly exhibiting ordinary, geometrically simple Jacobians with non-maximal angle rank), at the time of finishing that manuscript it was unclear as to what constraints the Newton polygon, Galois group, and angle rank impose on each other. 
We invite the reader to consult Tables 16--21 in the appendix of \cite{Dupuy2021} to see what sort of combinations can occur.
The reader may find it useful to compare our results to the cases found there. 

We report in this manuscript that the Galois group, Newton polygon, and angle rank are not independent, in that the \emph{permutation representation} of the Newton slopes actually determines the angle rank.
This subtlety was overlooked in \cite{Dupuy2021}.
Furthermore, certain combinations of Newton polygon and group are provably impossible (this was open at the time of writing \cite{Dupuy2021}). 
The subtlety in the data is that there are in fact many ways in which a group and a Newton polygon can be represented by a collection of roots; the fact that the Newton polygon and Galois group (as a permutation representation) determine the angle rank dates back at least to a letter from Serre to Ribet from 1981 \cite[pp.~6--11]{Serre2013} introducing the \emph{Frobenius torus} of an abelian variety. (This point of view was subsequently fleshed out in \cite{Chi1992}.) In this context the angle rank is one less than the dimension of the Frobenius torus.

In the present manuscript, the main object of study is the vector subspace $V$ of $\QQ^g$ spanned by $(v(\beta_1),\ldots,v(\beta_g))$ as $v$ varies over extensions of the $p$-adic valuation to $L$, the splitting field of $P(T)$. 
Here $\beta_i = \alpha_i/\alpha_{i+g}$ where $1\leq i \leq g$.
We refer to the system of hyperplanes defined by these vectors in the dual space as the \emph{Newton hyperplane arrangement}. 
Since each valuation can be obtained from the others by composing with a Galois element,
$V$ carries the structure of a $\Gal(L/\QQ)$-module; we call this the \emph{Newton hyperplane representation}.
\begin{remark}
	The authors arrived at this choice of subspace after playing with many other variants of this construction which are not included in this manuscript, so these choices may not appear obvious to the reader.
\end{remark}

In this manuscript, we report that the Newton hyperplane representation is a powerful hands-on tool to unify and generalize many known cases of the Tate conjecture for abelian varieties over finite fields with short proofs. In particular we prove the following:

\begin{theorem}[Generalized Lenstra--Zarhin]
Suppose that the Newton slopes of $A$ consist of $1/2$ with multiplicity $2$ and all other slopes in $\ZZ_{(2)}$. 
\begin{enumerate}
	\item[(a)]
	If $g$ is even, then $A$ has maximal angle rank.
	\item[(b)]
	If $g$ is odd, then $A$ has angle rank at least $g-1$.
\end{enumerate}
\end{theorem}
In case (a), the Tate conjecture is true for all powers of $A$.
This generalizes the previously known cases of the Tate conjecture where $g$ is even and $A$ is almost ordinary proved in \cite{Lenstra1993} (in a manner suggested in \cite[Remark~6.7]{Lenstra1993}).

Assume in what follows that $A$ is not supersingular.
The Galois group $G$ of $P(T)$ is an extension of $\Gbar\subset S_g$ by $C \subset \FF_2^g$.
We call $C$ the \emph{code} of $A$ and it is convenient to view $C$ as a collection of functions $\lbrace 1,2,\ldots,g\rbrace \to \FF_2$.  
In turn, we will call an element $c\in C$ a \emph{codeword}.
It turns out that there exists a unique partition of $\{1,\dots,g\}$ into subsets such that each codeword is constant on these subsets and the sets are maximal with this property.
In the theorem below we call this the \emph{level set partition}.
\begin{theorem}[Generalized Tankeev]
Let $m$ be the number of subsets in the level set partition.
If $g/m$ is prime, then $A$ has angle rank in $\{m, g-m, g\}$.
\end{theorem}
\noindent This generalizes Tankeev's theorem \cite{Tankeev1984} that if $g$ is prime then $\delta \in \lbrace 1, g-1,g\rbrace$.
Along the way we prove things like the fact that $C=\FF_2^g$ implies $\delta=g$
(see Corollary~\ref{remark: vector space dimension})
 and give a number of criteria for determining when $C=\FF_2^g$.

The code always contains complex conjugation,
which corresponds to the all-1s codeword.
We say that the code is \emph{nontrivial} if $\dim_{\F_2} C \geq 2$.

\begin{theorem}
Suppose that $\overline{G}$ acts primitively on $\{1,\dots,g\}$ and $C$ is nontrivial. Then $A$ has maximal angle rank. In particular, the Tate conjecture is true for all powers of $A$.
\end{theorem}

\smallskip

Finally, using the formalism of the Newton hyperplane vector space together with a variant of Siegel's lemma 
on integer solutions of linear systems of equations,
we give an effective version of a recent theorem of Zarhin.
Before stating the theorem, we fix some notation. 
Consider a relation among Frobenius eigenvalues
$$\alpha_1^{e_1}\alpha_2^{e_2} \cdots \alpha_{2g}^{e_{2g}} = \zeta q^d$$
where $e_i\geq 0$, $d = (e_1 + \cdots + e_{2g})/2$ is an integer, and $\zeta$ is a root of unity.
In what follows the \emph{weight} of the relation $(e_1,\ldots,e_{2g})$ is defined to be the $l^1$-norm $\vert e_1\vert + \cdots + \vert e_{2g} \vert$.

\begin{theorem}[Effective Zarhin]
		Let $A$ be an absolutely simple abelian variety of dimension $g$ over $\FF_q$ with Frobenius eigenvalues $\alpha_1,\ldots,\alpha_{2g}$, Galois group $G$, and angle rank $\delta$. 
	The integer vectors $(e_1,\ldots,e_{2g})\in \ZZ^{2g}$ such that $\alpha_1^{e_1}\cdots\alpha_{2g}^{e_{2g}} = \zeta q^d$ for some root of unity $\zeta$ and some integer $d$ form a group generated by vectors of weight less than or equal to 
	$$g m^{2\delta} \delta^{\delta+1},$$
	where $m$ denotes the least common denominator of the Newton slopes of $A$. 
\end{theorem}
In \cite{Zarhin2021}, Zarhin proves that there exists a constant as above but does not exhibit an explicit value.
To get a bound depending only on $g$, we may replace $\delta$ by $g$ and $m$ by $g!$; however, in practice this is a gross overestimate of $m$ (e.g., when $A$ is ordinary we have $m=1$). 

\subsection*{Acknowledgements} We thank Alexander Mueller, David Roe, Jean-Pierre Serre, Caleb Springer, Christelle Vincent, and Yuri Zarhin for helpful conversations. We thank James Milne for providing some references for Subsection \ref{ss:hyperpl-arrang}. 

We are grateful to the organizers of ``Emory Conference on Higher Obstructions to Rational Points'' (Atlanta, 2017) and ``Kummer Classes and Anabelian Geometry'' (Burlington, 2016) for providing stimulating environments where progress was made.

\section{Background and notation}

Throughout this paper, let $A$ be an absolutely simple abelian variety of dimension $g$ over a finite field $\FF_q$ of characteristic $p$. 
Denote the eigenvalues of Frobenius on $A$ by $\alpha_1,\dots,\alpha_{2g}$ with the convention that $\alpha_{i+g} = \overline{\alpha}_i$ for $i=1,\dots,g$;
we view these quantities as elements of the number field $L = \QQ(\alpha_1,\dots,\alpha_{2g})$ rather than as complex numbers (so that we may consider also $p$-adic embeddings).
The field $L$ is the Galois closure of the CM field $\QQ(\alpha_1)$;
let $L^+ = \QQ(\alpha_1 + \overline{\alpha}_1, \dots, \alpha_g + \overline{\alpha}_g) \subset L$ be the Galois closure of the totally real field $\QQ(\alpha_1 + \overline{\alpha}_1)$.

We also work with the quantities $\beta_i = \alpha_i^2/q = \alpha_i/\overline{\alpha}_i$.
Note that by swapping $\alpha_i$ with $\alpha_{i+g}$ in the original ordering of the eigenvalues, we can replace $\beta_i$ with $\overline{\beta}_i$ without changing $\beta_j$ for $j \neq i$.

\subsection{Angle rank}

Let $\Gamma$ be the multiplicative subgroup of $L^\times$ generated by $\alpha_1,\dots,\alpha_{2g}$ (which automatically contains $q$).
Let $\Gamma'$ be the subgroup of $\Gamma$ generated by $\beta_1, \dots, \beta_g$. Each of the complex absolute values of $L$ induces the same surjective homomorphism
$\Gamma \to (\sqrt{q})^\ZZ$ whose kernel contains $\Gamma'$ with finite index; we deduce that $\rank \Gamma' = \rank \Gamma - 1$. We refer to this quantity as the \emph{Frobenius angle rank} of $A$, or for short the \emph{angle rank};
it belongs to $\{0,\dots,g\}$. If the angle rank equals $g$, we also say that $A$ has \emph{maximal angle rank}.

The angle rank of $A$ is invariant under base change from $\FF_q$ to a finite extension (see \cite[\S3.4]{Dupuy2021}). 
It is easy to check (see \cite[Lemma~2.1]{Dupuy2021}) that $A$ has angle rank $g$ if and only if some base change of $A$ is \emph{neat} in the sense of Zarhin
\cite[\S 2]{Zarhin2015}. Another equivalent condition is that for every positive integer $m$, all of the Tate classes of $A^m$ (in all codimensions) are Lefschetz classes
\cite{Hazama1985}; this implies that the Tate conjecture holds for every power of $A$ \cite{Lenstra1993}.

Another interpretation of the angle rank is that it is one less than the rank of the \emph{Frobenius torus} associated to $A$ in the sense of Serre \cite[pp. 6--11]{Serre2013}.

\subsection{The Galois action on Frobenius eigenvalues}\label{sec:galois}

We will observe later (Remark~\ref{remark: low angle rank}) that $A$ has angle rank 0 if and only if it is supersingular, which can only occur for $g=1$. We omit this case in the following discussion; we may thus assume that $\alpha_1,\dots,\alpha_{2g}$ are distinct.

We identify $G \colonequals \Gal(L/\QQ)$ with a subgroup of the group $S_g^{\pm} \cong\FF_2^g \rtimes S_g$ of $g \times g$ signed permutation matrices
by mapping the element $h \in G$ to the matrix $M(h)$ with
\begin{equation}\label{eqn:m-matrix}
M(h)_{ij} = \begin{cases} 1, & h(\alpha_j) = \alpha_i \\
-1, & h(\alpha_j) = \overline{\alpha}_i \\
0, & \mbox{otherwise}.
\end{cases}
\end{equation}
This gives us the following commutative diagram with exact rows and injective vertical arrows, in which $\overline{G} \colonequals \Gal(L^+/\QQ)$
identified with a subgroup of $S_g$:
\begin{equation} \label{eq:imprimitivity sequence}
\xymatrix{
1 \ar[r] & C \ar[r] \ar[d] & G \ar[r] \ar[d] & \overline{G} \ar[r] \ar[d] & 1 \\
1 \ar[r] & \FF_2^g \ar[r] & S_g^{\pm} \ar[r] & S_g \ar[r] & 1
}
\end{equation}

\begin{remark}
The problem of classifying CM types was first studied by Dodson \cite{Dodson1984}.
The upper row of \eqref{eq:imprimitivity sequence} is what Dodson calls an \emph{imprimitivity sequence} for the group $G$;
Dodson further notes that from \eqref{eq:imprimitivity sequence} we obtain a 1-cocycle (crossed homomorphism) $\overline{G} \to \FF_2^g/C$.
\end{remark}

Because $A$ is absolutely simple, the Honda--Tate theorem \cite[Theorem~8]{Waterhouse1971} implies that the action of $\overline{G} \subseteq S_g$ on $\{1,\dots,g\}$ is \emph{transitive}, that is, the set $\{1,\dots,g\}$ forms a single $\overline{G}$-orbit.
Some possible stronger restrictions that can be imposed on this action include the following.
\begin{itemize}
\item
The action of $\overline{G}$ is \emph{primitive} if there is no subset $S$ of $\{1,\dots,g\}$ with $2 \leq \#S < g$ such that for any two $g,h \in \overline{G}$, $g(S)$ and $h(S)$ are either equal or disjoint.
If such an $S$ exists, it is called a \emph{block} or a \emph{domain of imprimitivity} for the action; the set of translates of $S$ is then itself a set with a $\overline{G}$-action.
\item
The action of $\overline{G}$ is \emph{$2$-transitive} if the induced action on ordered pairs of distinct elements of $\{1,\dots,g\}$ is transitive. Any 2-transitive action is also primitive.
\end{itemize}

We call the subgroup $C$ of $G$ the \emph{code} of $A$, and refer to elements of $C$ as \emph{codewords}; this is meant to suggest an interpretation of $C$ as a binary linear code of length $g$,
although we will not make use of any deep concepts from coding theory.
Note that the unique complex conjugation on $L$ corresponds to the all-ones vector in $C$; if this is the only nonzero codeword, we say that $C$ is \emph{trivial} (or that $A$ has \emph{trivial code}).

\section{The method of slope vectors}

\subsection{The slope vector subspace}

Let $V$ be the subspace of $\QQ^g$ spanned by the vectors $(v(\beta_1),\dots,v(\beta_g))$ as $v$ varies over all extensions of the $p$-adic valuation to $L$; as is customary with Newton polygons, valuations are normalized so that $v(q)=1$.
The reader should compare the next proposition to \cite[Lemma 3.6, (b) implies (c)]{Zywina2020}  and  \cite[Theorem 3.4]{Chi1992};
see also Remark~\ref{rmk:CM types}.
\begin{proposition} \label{lemma: angle rank from slopes}
The angle rank of $A$ is equal to $\dim_{\QQ} V$.
\end{proposition}
\begin{proof}
It suffices to check that for $e_1,\dots,e_g \in \ZZ$, $\beta_1^{e_1} \cdots \beta_g^{e_g}$ is a root of unity 
 if and only if 
\begin{equation}\label{eqn:newton-hyperplane}
e_1 v(\beta_1) + \cdots + e_g v(\beta_g) = 0,
\end{equation} 
for all extensions $v$ of the $p$-adic valuation to $L$. The former clearly implies the latter.

Suppose the latter holds. 
Let $\gamma:=\beta_1^{e_1} \cdots \beta_g^{e_g}$. 
We have $v(\gamma)=0$ for every $v$ above $p$ (by hypothesis). 
Also, $v(\gamma)=0$ for every place not $v$ not above $p$ (because each $\beta_i$ has trivial absolute value there).
This implies that $\gamma$ is not contained in any maximal ideal and hence is a unit. 
By the Riemann Hypothesis for varieties over finite fields, we know $\vert \psi(\gamma)\vert_{\infty}=1$ for every embedding $\psi\colon L\to \CC$.
By Kronecker's theorem, $\gamma$ is a root of unity.
\end{proof}

\begin{remark} \label{remark: low angle rank}
By Proposition~\ref{lemma: angle rank from slopes}, the angle rank of $A$ equals 0 if and only if $A$ is supersingular.
Since we are assuming $A$ is absolutely simple, this can only occur if $g=1$.

Similarly, the angle rank of $A$ is at most 1 if and only if $\beta_i/\beta_j$ is a root of unity for all $i,j$.
This means that the characteristic polynomial of Frobenius has the form $P(x^d)^{e}$ for some $d,e$; a typical example from the LMFDB is \href{https://www.lmfdb.org/Variety/Abelian/Fq/3/2/a\_a\_ac}{3.2.a\_a\_ac}.
\end{remark}

We assume hereafter that $A$ is not supersingular, so that \S\ref{sec:galois} applies.
We view $V$ as a $\QQ[G]$-module through the action described in \eqref{eqn:m-matrix}.
\begin{remark} \label{rmk:Newton matrix}
	We think of equation \eqref{eqn:newton-hyperplane} as defining a hyperplane in the variables $e_1,e_2,\ldots, e_g$, and have come to speaking of these as \emph{Newton hyperplanes}. 
	Collecting all of the coefficients into this matrix, one obtains a $\vert G \vert \times g$ matrix which we have come to call the \emph{Newton hyperplane matrix}. 
	The row space of the Newton hyperplane matrix is $V$, the \emph{Newton hyperplane representation}.
	See \S\ref{ss:hyperpl-arrang} for more discussion.
\end{remark}

\begin{remark}[Dodson's rank equals angle rank] \label{rmk:CM types}
Let $F$ be an algebraically closed field containing $\QQ$.
For $\widetilde{A}$ an abelian variety over $F$ with CM by an order in the field $K$, the action of this order on the tangent space of $\widetilde{A}$ defines a CM type $\Phi \subset \Hom(L, F)$ valued in $F$.
When $F = \overline{\QQ}_p$ and $\widetilde{A}$ is the lift of our abelian variety $A$ over $\FF_q$, we can express the Newton polygon of $A$ in terms of the CM type via the Shimura--Taniyama formula: for each $p$-adic place $v$ on $L$,
\[
v(\alpha) = \frac{\#\{\phi \in \Phi\colon \mbox{$\phi$ induces $v$ on $L$}\}}{[L_v:\QQ_p]}
\]
(e.g., see \cite[Proposition~2.1.4.2]{Chai2014}).

A consequence is that the \emph{rank} of $\widetilde{A}$ in the sense of
\cite[(3.1.0)]{Dodson1984} coincides with $\dim_{\QQ} V$, and hence by Proposition~\ref{lemma: angle rank from slopes} with the angle rank of $A$.
In particular, $A$ has maximal angle rank if and only if $\widetilde{A}$ is \emph{nondegenerate} in the sense of \cite{Dodson1984};
in other words, all of the Tate classes on powers of $A$ are Lefschetz classes if and only if all of the Hodge classes on powers of $\widetilde{A}$ are Lefschetz classes.  
\end{remark}

\subsection{Results of Tankeev and Lenstra--Zarhin}

The group $G$ acts on $V$ via its identification with a group of signed permutation matrices.
Some prior examples of the use of this action to control the angle rank are the following.

\begin{corollary}[Tankeev, \cite{Tankeev1984}] \label{corollary:tankeev}
If $g$ is prime, then $A$ has angle rank in $\{1, g-1, g\}$. 
\end{corollary}
\begin{proof}
This follows from Proposition~\ref{lemma: angle rank from slopes} because $G$ contains a cyclic group of order $g$ acting as a permutation of the coordinates, and as a representation of this group $\QQ^g$ decomposes as the direct sum of the diagonal subspace and the trace-zero subspace.
\end{proof}

The abelian variety $A$ is \emph{almost ordinary} if its Newton slopes consist of $0$ and $1$ each with multiplicity $g-1$ and $1/2$ with multiplicity 2. (The terminology
indicates that this condition corresponds to the unique codimension-1 stratum of the Newton polygon stratification on the moduli space of abelian varieties with a fixed polarization degree.)

\begin{corollary}[Lenstra--Zarhin, \cite{Lenstra1993}] \label{corollary:almost ordinary}
Suppose $A$ is almost ordinary.
\begin{enumerate}
\item[(a)]
If $g$ is even, then $A$ has maximal angle rank.
\item[(b)]
If $g$ is odd, then $A$ has angle rank at least $g-1$.
\end{enumerate}
\end{corollary}
The angle rank $g-1$ can indeed occur (even excluding the trivial case $g=1$); an example from the LMFDB is \href{https://www.lmfdb.org/Variety/Abelian/Fq/3.2.a\_ab\_ac}{3.2.a\_ab\_ac}.
\begin{proof}
We again use Proposition~\ref{lemma: angle rank from slopes} to reduce to computing the dimension of $V$.
For every extension $v$ of the $p$-adic valuation to $L$, the vector $(v(\beta_1),\dots,v(\beta_g)$) has one entry equal to 0 (corresponding to the slope $1/2$) and the rest equal to $\pm 1$.
Since $\overline{G}$ acts transitively, $V$ contains vectors over $\ZZ$ congruent modulo 2 to all of the rows of the matrix
\[
M = 
\begin{pmatrix}
0 & 1 & \cdots & 1 & 1 \\
1 & 0 & \cdots & 1 & 1 \\
\vdots & \vdots &\ddots & \vdots &\vdots \\
1 & 1 &\cdots & 0 & 1 \\
1 & 1 &\cdots & 1 & 0
\end{pmatrix}
\]
and so the dimension of $V$ can be bounded below by the rank of $M$ over $\FF_2$.
\end{proof}

The same argument yields the following generalization, as indicated in \cite[Remark~6.7]{Lenstra1993}.
\begin{corollary} \label{T:generalized LZ}
Suppose that the Newton slopes of $A$ consist of $1/2$ with multiplicity $2$ and all other slopes in $\ZZ_{(2)}$. 
\begin{enumerate}
\item[(a)]
If $g$ is even, then $A$ has maximal angle rank.
\item[(b)]
If $g$ is odd, then $A$ has angle rank at least $g-1$.
\end{enumerate}
\end{corollary}
An illustrative example is the Newton slopes
\[
\{1/3, 1/3, 1/3, 1/2, 1/2, 2/3, 2/3, 2/3\}.
\]
\begin{proof}
Under this hypothesis, $V$ again contains vectors over $\ZZ_{(2)}$ congruent modulo 2 to all of the rows of the matrix $M$ from the proof of Corollary~\ref{corollary:almost ordinary}. Thus the dimension of $V$ is again bounded below by the rank of $M$ over $\FF_2$.
\end{proof}

We also mention an easy complementary result.
\begin{corollary}
Suppose that the Newton slopes of $A$ consist of $0$ and $1$ each with multiplicity $1$ and $1/2$ with multiplicity $2(g-1)$. Then $A$ has maximal angle rank.
\end{corollary}
\begin{proof}
Under this hypothesis, $V$ contains a vector supported in a single coordinate. Since $\overline{G}$ acts transitively on $\{1,\dots,g\}$, this implies that $V = \QQ^g$.
\end{proof}

\subsection{Effect of the code}
We look at the effect of the code on the angle rank. 
We exploit the delicate property that, up to sign, the action of $\overline{G}$ on $C \subset \FF_2^g$ is the same as the action of $G$ on $V \subset \QQ^g$ through the signed permutation representation $h\mapsto M(h)$ described in Section \ref{sec:galois}.
This can be used to identify many new cases of maximal angle rank, including cases which are left ambiguous by the results of Tankeev and Lenstra--Zarhin.

\begin{lemma} \label{lemma: code partition}
There exists a unique partition
\[
\{1,\dots,g\} = \bigsqcup_{i=1}^m T_i
\]
such that each codeword of $C$ is constant on each $T_i$ and each $T_i$ is maximal for this property.
In particular, $C$ is trivial if and only if $m=1$.
(Note that this partition, being unique, must be preserved by $\overline{G}$.)
\end{lemma}
\begin{proof}
We take the $T_i$ to be the maximal nonempty subsets of $\{1,2,\dots,g\}$ with the property that each codeword of $C$ is constant on each $T_i$.
Since singleton sets have this property, the $T_i$ form a covering of $\{1,2,\dots,g\}$. To see that they are pairwise disjoint, note that if $T_i \cap T_j \neq \emptyset$,
then every codeword of $C$ is also constant on $T_i \cup T_j$; this is only possible if $T_i = T_j$.
\end{proof}

In what follows we identify $\QQ^g$ with $\QQ^{\oplus \lbrace 1,2,\ldots, g\rbrace}$ and write 
\[
\QQ^g = \bigoplus_{i=1}^m\QQ^{\oplus T_i}
\]
for the partition $T_1,\dots,T_m$ provided by Lemma~\ref{lemma: code partition}.
\begin{lemma} \label{lemma: Newton space partition}
For $i=1,\dots,m$, set $V_i = V \cap \QQ^{\oplus T_i}$. Then we have a decomposition $V = \bigoplus_{i=1}^m V_i$ of $\QQ$-vector spaces (without the action of $G$).
\end{lemma}
\begin{proof}
It is obvious that the sum of the $V_i$ is contained in $V$; it thus suffices to show that every element of $V$ can be decomposed into a sum over the $V_i$.
It further suffices to check that for any nonzero $v \in V$ whose support is minimal with respect to inclusion, we have $\supp v \subseteq V_i$ for some $i$. 
Suppose the contrary: then by Lemma~\ref{lemma: code partition}, there is a codeword $c \in C$ which is not constant on $\supp v$.
Then $v + c(v)$ is an element of $V$ with support $\supp v \setminus \supp c$, a proper nonempty subset of $\supp v$; this yields the desired contradiction.
\end{proof}

\begin{corollary} \label{remark: vector space dimension}
Suppose that $g>1$ and set $g' = g/m$. Then the angle rank of $A$ has the form $im$ for some $i \in \{1,\dots,g'\}$; in particular, it must be at least $m$.
\end{corollary}
\begin{proof}
In Lemma~\ref{lemma: Newton space partition}, the spaces $V_i$ are all carried to each other by various elements of $G$. We thus have $m | g$ and 
\[
\dim_\QQ V = m \dim_\QQ V_1.
\]
By Remark~\ref{remark: low angle rank}, for $g > 1$ we cannot have $V_1 = 0$, so $\dim_\QQ V_1 \in \{1,\dots,g'\}$.
\end{proof}

\begin{corollary} \label{corollary: rank g-1}
If the angle rank of $A$ equals $1$ or $g-1$, then $m = 1$ and so $C$ is trivial. In particular,
in either the theorem of Tankeev (Corollary~\ref{corollary:tankeev}) or the theorem of Lenstra--Zarhin (Corollary~\ref{corollary:almost ordinary}), 
if $C$ has nontrivial code, then $A$ has maximal angle rank.
\end{corollary}

On a related note, we have the following.
\begin{theorem} \label{theorem: primitive to maximal angle rank}
Suppose that $\overline{G}$ acts primitively on $\{1,\dots,g\}$ and $C$ is nontrivial. Then $A$ has maximal angle rank.
\end{theorem}
\begin{proof}
Note that nontriviality of $C$ forces $g>1$.
Since $\overline{G}$ is primitive, the only way it can preserve the partition from Lemma~\ref{lemma: code partition} is to have $m=1$ or $m=g$.
The case $m=1$ corresponds to $C$ being trivial, which we are forbidding here.
If $m=g$, then by Corollary~\ref{remark: vector space dimension} the angle rank must equal $g$.
\end{proof}

\begin{corollary}
	Suppose $g$ is prime. 
	If $C$ is nontrivial, then $A$ has maximal angle rank. 
\end{corollary}
\begin{proof}
The block size divides $g$, so $\overline{G}$ acts primitively and Theorem \ref{theorem: primitive to maximal angle rank} applies.
\end{proof}

\subsection{A further reduction}

We saw in Theorem~\ref{theorem: primitive to maximal angle rank} above that when $\overline{G}$ acts primitively, non-maximal angle rank implies trivial code.
We next describe a further reduction that allows one to apply similar logic in the case when $\overline{G}$ does not act primitively.

Let $\widetilde{H}$ be the subgroup of $G$ stabilizing $T_1$.
Let $H$ be the image of $\widetilde{H}$ in $\GL(\QQ^{\oplus T_1})$, again viewed as a group of signed permutation matrices.
If $g>1$ and $H$ acts irreducibly on $\QQ^{\oplus T_1}$, then $V_1 = \QQ^{\oplus T_1}$ and $A$ must have maximal angle rank.

We observe that $H$ fits into an exact sequence
\begin{equation} \label{eq:reduced code}
1 \to C_H \to H \to \overline{H} \to 1,
\end{equation}
where $\overline{H}$ denotes the permutation group obtained from $H$ by ignoring signs and $C_H$ is the projection of $C$ into $\GL(\QQ^{\oplus T_1})$.
Note that $C_H$ cannot be zero (or else $C$ would be zero, whereas it contains the all-ones vector) but every vector of $C_H$ is constant, so
$C_H \cong \ZZ/2\ZZ$.

\begin{lemma} \label{lemma: one-dimensional subspace}
Suppose that there exists a one-dimensional subspace $W$ of $\QQ^{\oplus T_1}$ stable under the action of $H$. (For example, this holds if $\dim_\QQ V_1 \in \{1, g' - 1\}$.)
\begin{enumerate}
\item[(a)]
The exact sequence \eqref{eq:reduced code} splits, so that $H \cong \ZZ/2\ZZ \times \overline{H}$.
\item[(b)]
The subspace $W$ is generated by some vector with components in $\{\pm 1\}$, which is stable under the action of $H$.
\end{enumerate}
\end{lemma}
\begin{proof}
Let $v$ be a generator of the subspace in question. 
For each $h \in H$, we then have $h(v) = f(h) v$ for some $f(h) \in \{\pm 1\}$; the map $f\colon H \to \{\pm 1\}$ 
is a group homomorphism which splits \eqref{eq:reduced code}; this yields (a). Part (b) is a consequence of the transitivity of the action of $\overline{H}$ on $T_1$.
\end{proof}

\begin{remark} \label{remark: vector space dimension2}
In Corollary~\ref{remark: vector space dimension}, the minimal case is when $\dim_\QQ V_1 = 1$. 
By Lemma~\ref{lemma: one-dimensional subspace}, $H \cong \ZZ/2\ZZ \times \overline{H}$
and we can relabel the Frobenius eigenvalues so that $V_1$ is the span of the all-ones vector on $T_1$.
For each $h \in \overline{H}$, $h(\beta_1)/\beta_1$ is a root of unity by Kronecker's theorem; this implies that $\overline{H}$ is contained in a copy of the affine group $\ZZ/g'\ZZ \rtimes (\ZZ/g'\ZZ)^\times$.
\end{remark}

\begin{remark} \label{remark: vector space dimension3}
In Remark~\ref{remark: vector space dimension}, suppose that $\dim_\QQ V_1 = g' - 1$.
By Lemma~\ref{lemma: one-dimensional subspace}, $H \cong \ZZ/2\ZZ \times \overline{H}$
and we can relabel the Frobenius eigenvalues so that $V_1$ is the trace-zero subspace of $\QQ^{\oplus T_1}$.
In this case the relations in $\Gamma'$ are generated by
\[
\prod_{j \in T_i} \beta_j = 1 \qquad (i=1,\dots,m).
\]
\end{remark}

By analogy with Tankeev's theorem (Corollary~\ref{corollary:tankeev}), we have the following.
\begin{theorem} \label{theorem:tankeev-reduced}
If $g'$ is prime, then $A$ has angle rank in $\{m, g-m, g\}$. 
\end{theorem}
\begin{proof}
Since $\overline{H}$ acts transitively on $T_1$, $H$ must contain a cyclic group of order $g'$ acting as a signed permutation of the coordinates. As a representation of this group, $\QQ^{\oplus T_1}$ decomposes as the direct sum of the diagonal subspace and the trace-zero subspace.
\end{proof}

When $g'$ is not prime, we can sometimes apply the following result.
\begin{theorem} \label{T:nonmaximal angle rank with 2-transitive group}
Suppose that $g>1$, $g' \neq 3$, $\overline{H}$ acts $2$-transitively on $T_1$, and $A$ does not have maximal angle rank. Then the angle rank of $A$ is $g-m$.
\end{theorem}
\begin{proof}
The cases $g' = 1,2$ are covered by Remark~\ref{remark: vector space dimension}, so we may assume $g' \geq 4$.
Let $\overline{\chi}\colon \overline{H} \to \ZZ$ be the character of the permutation representation. 
The inner self-product of $\overline{\chi}$ in the representation ring of $\overline{H}$ equals
\begin{align*}
\frac{1}{\#\overline{H}} \sum_{h \in \overline{H}} \overline{\chi}(h)^2 &= \frac{1}{\#\overline{H}}\#\{(h,i,j) \in \overline{H} \times T_1 \times T_1: \overline{h}(i) = i, \overline{h}(j) = j\} \\
&= \frac{1}{\#\overline{H}}  \#\{(h,i) \in \overline{H} \times T_1: \overline{h}(i) = i \}  \\
&\qquad +
\frac{1}{\#\overline{H}} \#\{(h,i,j) \in \overline{H} \times T_1 \times T_1: \overline{h}(i) = i, \overline{h}(j) = j, i \neq j \}  \\
&= 1 + 1 = 2,
\end{align*}
where the last line follows from the previous one by 2-transitivity.
Hence the permutation representation splits as the direct sum of the trivial representation and an absolutely irreducible representation.

Let $\chi\colon H \to \ZZ$ be the character of the linear representation of $H$ on $V_1$. The inner self-product of $\chi$ in the representation ring of $H$ equals
\[
\frac{1}{\#H} \sum_{h \in H} \chi(h)^2 = \frac{1}{\#H} \sum_{\overline{h} \in \overline{H}} (\chi(h_1)^2 + \chi(h_2)^2) 
\]
where $h_1, h_2 \in H$ are the two lifts of $\overline{h} \in \overline{H}$. Note that $\chi(h_1)^2 = \chi(h_2)^2 \leq \chi(\overline{h})^2$;
we conclude from this that the inner self-product of $\chi$ is at most 2.

Since $A$ does not have maximal angle rank, $H$ does not act irreducibly on $\QQ^{\oplus T_1}$. 
In this case the inner self-product of $\chi$ must be exactly 2. 

For the linear action of $H$ on $\End(\ZZ_{(2)}^{\oplus T_i})$, the $H$-fixed subspace is a $\ZZ_{(2)}$-lattice in the $H$-fixed subspace of $\End(\QQ^{\oplus T_1})$ and therefore is free of rank 2.
One evident element of this lattice is the identity element for composition
\[
e \colonequals \sum_{i \in T_1} e_i^* \otimes e_i \in \End(\ZZ_{(2)}^{\oplus T_1}).
\]
Meanwhile, the reduction of this lattice modulo 2 is a 2-dimensional vector space over $\FF_2$, which injects into the $\overline{H}$-fixed subspace of $\End(\FF_2^{\oplus T_1})$ which is also of dimension 2. This equality of dimensions implies that the $\overline{H}$-fixed vector 
\[
\sum_{i,j \in T_1; i \neq j} e_i^* \otimes e_j \in \End(\FF_2^{\oplus T_1})
\]
must also lift to an $H$-fixed vector of the form
\[
v = \sum_{i,j \in T_1; i \neq j} a_{ij} e_i^* \otimes e_j \in \End(\ZZ_{(2)}^{\oplus T_1}).
\]
Since $\overline{H}$ acts 2-transitively, we have $a_{ij} = \pm a_{i'j'}$ for any two pairs $i \neq j$ and $i' \neq j'$; we may thus assume that $a_{ij} \in \{\pm 1\}$ for all $i,j$.

Set $a_{ii} \colonequals 0$ for $i \in T_1$. Since the composition map on $\End(\ZZ_{(2)}^{\oplus T_1})$ is $H$-equivariant,
\[
v \circ v = \sum_{i,j,k \in T_1} a_{ij} a_{jk} e_i^* \otimes e_k
\]
is also an $H$-invariant element of $\End(\ZZ_{(2)}^{\oplus T_1})$, and so has the form $c_0 e + c_1 v$ for some $c_1, c_2 \in \QQ$. 
That is,
\[
\sum_{j \in T_1} a_{ij} a_{jk} = \begin{cases} c_0 & (i = k) \\
c_1 & (i \neq k);
\end{cases}
\]
in particular, $c_0, c_1$ are integers with $c_0 \equiv g'-1 \pmod{2}$, $c_1 \equiv g' \pmod{2}$, and $|c_0|, |c_1| < g'$.

The possible eigenvalues of $v$ are $-c_1/2 \pm \sqrt{c_0 + c_1^2/4}$ with some multiplicities $i_{\pm}$.
Since $v$ has trace 0 by construction, we must have
\begin{align*}
i_+ + i_- &= g' \\
i_+ - i_- &= \frac{c_1 g'}{\sqrt{4c_0 + c_1^2}}.
\end{align*}
In particular $\sqrt{4c_0 + c_1^2}$ is rational; hence the possible eigenvalues of $v$, being both rational numbers and algebraic integers, must themselves be integers.

Now note that $v$ has at most two distinct integral eigenvalues, but it reduces modulo 2 to an endomorphism with rank $g'-1$. Hence $v$ must have two distinct eigenvalues, one an even integer occurring with multiplicity 1, the other an odd integer occurring with multiplicity $g'-1$.
From this it follows that $\dim V_1 \in \{0, 1, g'-1\}$. By Remark~\ref{remark: low angle rank}, we see that $\dim V_1 \neq 0$.
Since $\ZZ/g' \ZZ \rtimes (\ZZ/g'\ZZ)^\times$ does not act 2-transitively for $g' \geq 4$, by Remark~\ref{remark: vector space dimension2} we see that $\dim V_1 \neq 1$.
This proves the claim.
\end{proof}

\begin{remark}
To illustrate Theorem~\ref{T:nonmaximal angle rank with 2-transitive group},
consider the case where $g = 12, g' = 6$, and $\overline{H} = A_5$ acting on $\mathbb{P}^1(\FF_5)$ via the identification $A_5 \cong \mathrm{PSL}(2, 5)$ (i.e., the result of precomposing the standard permutation action of $S_6$ by an outer automorphism, then restricting to $A_5$).
In this case, if $A$ does not have maximal angle rank, then its angle rank must equal 10.
It would be enlightening to find an explicit abelian variety realizing this example, say over $\FF_2$, but this requires a computationally expensive expansion of the LMFDB tables which we do not wish to pursue. 
\end{remark}

\begin{remark}
The complete classification of 2-transitive permutation groups can be given in terms of the classification of finite simple groups; see \cite[\S 7.7]{Dixon1996}.
Absent the 2-transitive condition, 
it is possible for the permutation representation of $\overline{H}$ to give rise to a representation of $\QQ^{\oplus T_1}$ for which the trace-zero subspace is irreducible but not absolutely irreducible; see \cite{Dixon2005} for some discussion of this phenomenon. We do not know if this means that the representation of $H$ on $\QQ^{\oplus T_1}$ splits into at most two irreducibles.
\end{remark}

\subsection{A theorem of Zarhin}

We now reprove
\cite[Theorem 1.4]{Zarhin2021} in our language; the proof we give is really his, except that we paraphrase using our terminology
and also obtain an explicit constant. For the latter, we need a version of Siegel's lemma for overdetermined systems of equations. 

\subsubsection{A version of Siegel's lemma for overdetermined systems} \label{subsubsec:Siegel}
For an underdetermined system of equations $Bx=0$ with integer coefficients, with $M$ independent equations in $N$ unknowns with $N>M$, Bombieri and Vaaler's version \cite[Theorem 1]{Bombieri1983} of Siegel's lemma tells us there exists an integer solution $x\in \ZZ^N$ such that 
$$\Vert x\Vert_{\infty} \leq \left ( \dfrac{\sqrt{\det(BB^t})}{D}\right)^{1/(N-M)},$$
where $D$ is the gcd of the $M\times M$ minors of $B$.

We are concerned with a problem of the opposite form: we take $B$ to be the Newton hyperplane matrix, for which $M=\vert G\vert$ is always greater than or equal to $N = g$ (i.e., the system is overdetermined), and we want to produce generators of the right nullspace of $B$ which are ``as small as possible'' (here measured using the $l^1$-norm). 
In this case the minors of $B$ do not play a role, but since $B$ has entries in $\QQ$ we must account for the lcm $d$ of the denominators of the entries (each of which is bounded by $g$). 

\begin{lemma}\label{L:kernel-bounds}
	Let $B$ be a nonzero $M\times N$ matrix of rank $r$ whose entries are rational numbers with denominators dividing $d$. 
	Then the integer solutions of $Bx=0$ form a group generated by elements $x\in \ZZ^N$ with 
	$$ \Vert x \Vert_1 \leq N r\left( \sqrt{r} d\Vert B \Vert_\infty \right)^{2r},$$
	where $\Vert B \Vert_\infty := \max_{i,j} \vert B_{ij} \vert$.
\end{lemma}
\begin{proof}
Let $\widetilde{B} := dB$ be the matrix obtained from $B$ by
clearing denominators.
By applying \cite[Proposition 6.6]{storjohann2000algorithms}
%https://cs.uwaterloo.ca/~astorjoh/diss2up.pdf
to the transpose of $\widetilde{B}$,
we obtain a matrix $C$ whose columns span the right nullspace of $\widetilde{B}$ satisfying
	$$\Vert C \Vert_\infty \leq r\beta^2,$$
	where $\beta := (\sqrt{r} \Vert \widetilde{B} \Vert_\infty)^r \geq 1$ 	(this is where the hypothesis of $B$ being nonzero is used).
	Then the $l^1$-norm of each column of $C$ is bounded by
		 $$N \Vert C \Vert_\infty \leq N r( \sqrt{r} \Vert \widetilde{B} \Vert_\infty)^{2r} = N r( \sqrt{r} d \Vert B \Vert_\infty)^{2r}, $$
	yielding the desired estimate.
\end{proof}

\subsubsection{An effective version of a result of Zarhin}

We next apply the previous lemma to determine a bound on the weights ($l^1$-norms) of relations which generate all relations among Frobenius eigenvalues. 

\begin{theorem}
	Let $A$ be an absolutely simple abelian variety of dimension $g$ defined over $\FF_q$ with Frobenius eigenvalues $\alpha_1,\ldots,\alpha_{2g}$, Galois group $G$, and angle rank $\delta$. 
	The group of integer vectors $(e_1,\ldots,e_{2g})\in \ZZ^{2g}$ such that $\alpha_1^{e_1}\cdots\alpha_{2g}^{e_{2g}} = \zeta q^d$ for some root of unity $\zeta$ and some integer $d$ (depending on the vector)
	is generated by its elements of $l^1$-norm at most
	$$g m^{2\delta} \delta^{\delta+1},$$
	where $m$ denotes the least common denominator of the Newton slopes of $A$. 
\end{theorem}
\begin{proof}
Under our running notations, we have 
$\alpha_{i+g} = \overline{\alpha}_i$ for $i=1,\dots,g$;
it thus suffices to work in the subgroup of relations of the form
\begin{equation} \label{E:relation2}
\alpha_{1}^{e_1}\cdots \alpha_g^{e_g}=\zeta q^d.
\end{equation}
From the proof of Lemma~\ref{lemma: angle rank from slopes},
\eqref{E:relation2} holds for some $d$ if and only if $(e_1,\dots,e_g)$ satisfies \eqref{eqn:newton-hyperplane} for every place $v$ of $L$ above $p$, i.e., $(e_1,\dots,e_g)$ is in  the kernel of the Newton hyperplane matrix (see  Remark~\ref{rmk:Newton matrix}). 
We now apply Lemma~\ref{L:kernel-bounds} to this matrix: in the notation of the lemma, $M, N, r, d$ correspond to our present $|G|, g, \delta, m$, while $\Vert B \Vert|_\infty \leq 1$ because the entries of $B$ are all in the interval $[-1,1]$.
We deduce that the group of relations is generated by its elements of $l^1$-norm at most $g \delta (\sqrt{\delta} m)^{2\delta}$, as claimed.
\end{proof}

\subsection{Hyperplane arrangements}
\label{ss:hyperpl-arrang}
With notation as above, for each extension $v$ of the $p$-adic valuation to $L$, consider the hyperplane in $\RR^g$ normal to the vector
$(v(\beta_1),\dots,v(\beta_g))$. These form a finite union of hyperplanes of $\RR^g$ (i.e., a \emph{hyperplane arrangement}) carrying a transitive action of $G$.
We thus also obtain an action on $G$ on the set of connected components of the complement of the hyperplane arrangement.
One can then apply a construction of Hazama \cite{Hazama2000, Hazama2003} to recover the CM abelian variety $\widetilde{A}$ associated to the hyperplane arrangement, which is a lift of $A$.

This provides a mechanism for translating back and forth between results about angle ranks of abelian varieties over finite fields and Hodge classes of CM abelian varieties (compare Remark~\ref{rmk:CM types}).
Using this mechanism, Milne \cite[Theorem~7.1]{Milne1999} showed that the Tate conjecture for all CM abelian varieties over $\CC$ implies the Tate conjecture for all abelian varieties over finite fields. By contrast, if we choose a single CM abelian variety $A$ with good reduction to $A_0$, it is not known that the Hodge conjecture for all powers of $A$ implies that Tate conjecture for all powers of $A_0$; however, see \cite[Theorem~1.5]{Milne2001} for a result of this form under some additional restrictions on $A$.

\bibliographystyle{amsalpha}
\bibliography{bizz}

\end{document}